\documentclass[11pt,reqno]{amsart}
\usepackage{amsmath,amssymb,geometry,color,tikz-cd}
\usepackage[backref,pagebackref,pdftex,hyperindex]{hyperref}
\usepackage[alphabetic]{amsrefs}
\usepackage{amssymb}
\usepackage{bbm}
\usepackage{enumitem}
\usepackage{upgreek}

\newtheorem{proposition}{Proposition}[section]
\newtheorem{theorem}[proposition]{Theorem}
\newtheorem{lemma}[proposition]{Lemma}

\theoremstyle{definition}

\newtheorem{definition}[proposition]{Definition}

\title{Product theorem on delta invariants via adding a general boundary}
\author{Chuyu Zhou}
\address{\'Ecole Polytechnique F\'ed\'erale de Lausanne (EPFL), MA C3 615, Station 8, 1015 Lausanne, Switzerland}
\email{chuyu.zhou@epfl.ch}
\date{} 

\newcommand{\ord}{{\rm {ord}}}

\newcommand{\vol}{{\rm {vol}}}
\newcommand{\lct}{{\rm {lct}}}

\newcommand{\dt}{{\rm {dt}}}

\newcommand{\Ric}{{\rm {Ric}}}

\newcommand{\bC}{\mathbb{C}}

\newcommand{\bQ}{\mathbb{Q}}

\begin{document}

\begin{abstract}
It's well-known that adding a general boundary would create K-stability. As an application, we reprove product theorem for delta invariants of Fano varieties.
\end{abstract}

\maketitle
\tableofcontents

\section{Introduction}

It's well known that a K-unstable Fano manifold cannot admit K$\rm{\ddot{a}}$hler-Einstein (KE) metrics. However, it can admit conic KE metrics along 
some smooth divisors. This inspires the concept of twisted K-stability.  In \cite{Der16},  Dervan introduced the concept of twisted K-stability and later it is algebraically reformulated using twisted generalized Futaki invariants in \cite{BLZ19}, which indicates that K-unstable Fano manifolds can be twisted K-stable. For example, if $X$ is a given K-unstable Fano manifold and $H\in |-lK_X|$ is a general smooth divisor on $X$, where $-lK_X$ is a very ample line bundle, then the pair $(X,\frac{1-\beta}{l}H)$ is uniformly K-stable for sufficiently small $0<\beta\ll 1$.  Thus there exists a conic KE metric $w(\beta)$ such that the following equation holds:
\begin{align*}
\Ric(w(\beta))= \beta w(\beta) + (1-\beta) [H].
\end{align*}

The following theorem proved by \cite{BL18b} are natural products of this phenomenon, that is, adding a general boundary would create K-stability.

\begin{theorem}{\rm (\cite[Section 7]{BL18b})}\label{kus}
Let $(X,\Delta)$ be a log Fano pair, then the following statements are equivalent.
\begin{enumerate}
\item $\delta(X,\Delta)\geq \mu$ for some positive $0<\mu\leq 1$,
\item For any rational $0<\epsilon <\mu$, there exists an element $D \in |-K_X-\Delta|_\bQ$ such that the pair $(X, \Delta+(1-\epsilon)D)$ is K-semistable,
\item For any rational $0<\epsilon <\mu$, there exists an element $D \in |-K_X-\Delta|_\bQ$ such that the pair $(X, \Delta+(1-\epsilon)D)$ is uniformly K-stable.

\end{enumerate}

\end{theorem}

As an application of this principle, we reprove the following product theorem for delta invariants of Fano varieties, which is originally proved by \cite{zhuang19}.

\begin{theorem}\label{corollary}
Let $(X_i,\Delta_i), i=1,2,$ be two log Fano pairs.
\begin{enumerate}
\item Both pairs are K-semistable if and only if $(X_1\times X_2,p_1^*\Delta_1+p_2^*\Delta_2)$ is K-semistable,
\item If one of the two pairs is K-unstable, then 
$$\delta(X_1\times X_2, p_1^*\Delta_1+p_2^*\Delta_2)=\min\{\delta(X_1,\Delta_1), \delta(X_2,\Delta_2)\}.$$
\item In general, we have following inequality,
$$\delta(X_1\times X_2, p_1^*\Delta_1+p_2^*\Delta_2)\geq  \min\{\delta(X_1,\Delta_1), \delta(X_2,\Delta_2), 1\}.$$
\end{enumerate}
\end{theorem}

\noindent
\subsection*{Acknowledgement}
The author would like to thank Yuchen Liu and Ziquan Zhuang for beneficial comments.
The author is supported by grant European Research Council (ERC-804334).

\section{Preliminaries}

In this section, we present some necessary preliminaries. Throughout the note we work over complex number field $\bC$. A log pair $(X,\Delta)$ consists of a normal projective variety $X$ and an effective $\bQ$-divisor $\Delta$ on $X$ such that $K_X+\Delta$ is $\bQ$-Cartier. We say a log pair $(X,\Delta)$ is log Fano if the pair admits klt singularities and $-K_X-\Delta$ is ample. The $\bQ$-linear system $|-K_X-\Delta|_\bQ$ is defined as follows:
 $$|-K_X-\Delta|_\bQ:= \{ D\geq 0|  D\sim_\bQ -K_X-\Delta\}.$$

By the works of \cite{FO18, BJ20}, we can give the following definition for delta invariants of log Fano pairs.

\begin{definition}
Let $(X,\Delta)$ be a log Fano pair. We define
$$\delta_m(X,\Delta):=\inf_E \frac{A_{X,\Delta}(E)}{S_m(E)} \quad and \quad \delta(X,\Delta):=\inf_E\frac{A_{X,\Delta}(E)}{S_{X,\Delta}(E)},$$
where $E$ runs over all prime divisors over $X$.  Here 
$$A_{X,\Delta}(E)=\ord_E(K_Y-f^*(K_X+\Delta))+1 $$
for some log resolution $f: Y\to X$ such that $E\subset Y$; and
$$S_m(E)=\sup_{D_m} \ord_E(D_m), $$
where $D_m$ is of the form $\frac{\sum_j\{s_j=0\}}{m\dim H^0(X,-m(K_X+\Delta))}$ and $\{s_j\}_j$ is a complete basis of the vector space $H^0(X,-m(K_X+\Delta))$; and 
$$S_{X,\Delta}(E)=\frac{1}{\vol(-K_X-\Delta)}\int_0^\infty \vol(-f^*(K_X+\Delta)-tE)\dt .$$
\end{definition}

The following result is now well-known. 

\begin{theorem}{(\rm \cite{FO18,BJ20})}
Let $(X,\Delta)$ be a log Fano pair, then
\begin{enumerate}
\item $\lim_{m\to \infty} \delta_m(X,\Delta)=\delta(X,\Delta)$,
\item $(X,\Delta)$ is K-semistable if and only if $\delta(X,\Delta)\geq 1$,
\item $(X,\Delta)$ is uniformly K-stable if and only if $\delta(X,\Delta)> 1$.
\end{enumerate}

\end{theorem}

One can just put this theorem as the definition of K-semistability (resp. uniform K-stability) of $(X,\Delta)$.

\section{Product theorem for K-stability}

In this section, we will prove Theorem \ref{corollary}. We first show the following result on K-stability of product varieties.
 
\begin{theorem}\label{kss product}
Let $(X_i,\Delta_i), i=1,2,$ be two log Fano pairs, then
$(X_1\times X_2, p_1^*\Delta_1+p_2^*\Delta_2)$ is K-semistable if and only if both pairs are K-semistable.
\end{theorem}

\begin{proof}
The only if direction is easy. Suppose $(X_1\times X_2, p_1^*\Delta_1+p_2^*\Delta_2)$ is K-semistable. For any $m$-basis type divisor $D^{(i)}_m$ for $(X_i, \Delta_i)$, we see that $p_1^*D^{(1)}_m+p_2^*D^{(2)}_m$ is an $m$-basis type divisor for $(X_1\times X_2, p_1^*\Delta_1+p_2^*\Delta_2)$. Denote $\delta_m:=\delta_m(X_1\times X_2, p_1^*\Delta_1+p_2^*\Delta_2)$, then the pair
 $$(X_1\times X_2, p_1^*\Delta_1+p_2^*\Delta_2+\delta_m(p_1^*D^{(1)}_m+p_2^*D^{(2)}_m))$$
 is log canonical, so are  $(X_1,\Delta_1+\delta_m D^{(1)}_m)$ and $(X_2,\Delta_2+\delta_m D^{(2)}_m)$. Thus $\delta_m(X_i,\Delta_i)\geq \delta_m$, which implies that both $(X_i,\Delta_i),i=1,2,$ are K-semistable.

For the converse direction, we just apply Theorem \ref{kus} to the case $\mu=1$.  For any fixed rational $0<\epsilon <1$, there exist $D_i\in |-K_{X_i}-\Delta_i|_\bQ, i=1,2$ such that the pairs 
$$(X_1,\Delta_1+(1-\epsilon)D_1) \quad and \quad (X_2,\Delta_2+(1-\epsilon)D_2)$$
are both uniformly K-stable by Theorem \ref{kus} (3). By \cite{LTW19}, one can construct KE metrics on the pairs $(X_i, \Delta_i+(1-\epsilon)D_i), i=1,2,$. Therefore, by taking product, the pair
$$(X_1\times X_2, p_1^*\Delta_1+p_2^*\Delta_2+(1-\epsilon)(p_1^*D_1+p_2^*D_2))$$
also admits a KE metric, hence is K-semistable (even K-polystable) by \cite{Ber16}. Again by Theorem \ref{kus} (2), we see the pair $(X_1\times X_2, p_1^*\Delta_1+p_2^*\Delta_2)$ is K-semistable. The proof is finished.
\end{proof}

If we assume one of the two pairs is K-unstable, then we have following more precise result on delta invariant of the product variety.

\begin{theorem}\label{product formula}
Let $(X_i,\Delta_i), i=1,2,$ be two log Fano pairs. Suppose one of the two pairs is K-unstable, then 
$$\delta(X_1\times X_2, p_1^*\Delta_1+p_2^*\Delta_2)=\min\{\delta(X_1,\Delta_1), \delta(X_2, \Delta_2)\}.$$
\end{theorem}

\begin{proof}
We assume $\delta(X_1,\Delta_1)\leq \delta(X_2, \Delta_2)$. We first show  
$$\delta(X_1\times X_2, p_1^*\Delta_1+p_2^*\Delta_2)\geq \delta(X_1,\Delta_1).$$
We apply Theorem \ref{kus} (3). For any rational $0<\epsilon <\delta(X_1,\Delta_1)$, one can find an element $D_1\in |-K_{X_1}-\Delta_1|_\bQ$ and an element $D_2\in |-K_{X_2}-\Delta_2|_\bQ$ such that both $(X_i,\Delta_i+(1-\epsilon)D_i)$ are uniformly K-stable.  By Theorem \ref{kss product}, the pair
$$(X_1\times X_2, p_1^*\Delta_1+p_2^*\Delta_2+(1-\epsilon)(p_1^*D_1+p_2^*D_2)) $$
is K-semistable. By Theorem \ref{kus} (2) we see that
$$\delta(X_1\times X_2, p_1^*\Delta_1+p_2^*\Delta_2)\geq \delta(X_1,\Delta_1). $$
We next show  $\delta(X_1\times X_2, p_1^*\Delta_1+p_2^*\Delta_2)\leq \delta(X_1,\Delta_1).$ It suffices to show that 
$$\delta_m(X_1\times X_2, p_1^*\Delta_1+p_2^*\Delta_2)\leq \delta_m(X_1,\Delta_1). $$
Similar to the proof of Theorem \ref{kss product}, we arbitrarily choose $m$-basis type divisors $D^{(i)}_m$ for $(X_i,\Delta_i)$, then $p_1^*D^{(1)}_m+p_2^*D^{(2)}_m$ is an $m$-basis type divisor for $(X_1\times X_2, p_1^*\Delta_1+p_2^*\Delta_2)$. Denote $\delta_m:=\delta_m(X_1\times X_2, p_1^*\Delta_1+p_2^*\Delta_2)$, we see the pair
$$(X_1\times X_2, p_1^*\Delta_1+p_2^*\Delta_2+\delta_m(p_1^*D^{(1)}_m+p_2^*D^{(2)}_m)) $$
is log canonical, so are the pairs $(X_i,\Delta_i+\delta_mD^{(i)}_m), i=1,2.$ Therefore we have
$$\delta_m(X_1,\Delta_1)\geq \delta_m. $$
Taking the limit we have
$$\delta(X_1,\Delta_1)\geq \delta(X_1\times X_2, p_1^*\Delta_1+p_2^*\Delta_2). $$
The proof is finished.

\end{proof}

\begin{proof}[Proof of Theorem \ref{corollary}]
The proof is just a combination of Theorem \ref{kss product} and Theorem \ref{product formula}.
\end{proof}

\bibliography{reference.bib}
\end{document}